\documentclass[12pt]{amsart}
\usepackage{amssymb}
\usepackage{mathrsfs}
\theoremstyle{plain}

\newtheorem{theorem}{Theorem}[section]
\newtheorem{lemma}[theorem]{Lemma}

\theoremstyle{definition}

\newcommand{\eps}{\varepsilon}

\newcommand{\NN}{\mathbb N}

\newcommand{\SSS}{\mathcal S}

\def\e{\varepsilon}

\def\al{\alpha}
\def\phi{\varphi}
\def\NN{{\mathbb N}}
\def\SS{{\mathcal S}}

\newif\ifComplain
\Complaintrue           
\def\complain#1{\ifComplain\ifhmode \newline\fi{\sf *** \ \ #1
\\}\fi}

\newif\ifmarglab
\makeatletter
\def\label#1{\@bsphack\ifmarglab\marginpar{LAB:#1}\fi\if@filesw {\let\thepage\relax
   \def\protect{\noexpand\noexpand\noexpand}%
   \edef\@tempa{\write\@auxout{\string
      \newlabel{#1}{{\@currentlabel}{\thepage}}}}%
   \expandafter}\@tempa
   \if@nobreak \ifvmode\nobreak\fi\fi\fi\@esphack}
\makeatother
\marglabfalse

\usepackage{graphicx}

\long\def\onefigure#1#2{
\begin{figure*}[tbp]
\begin{center}
#1
\end{center}
\caption{#2}
\end{figure*}
} 

\newcommand\newipefig[2]
{\onefigure{\includegraphics{#1.pdf}}{\label{#1} #2}}

\begin{document}

\title[Three projections]{Strange products of projections
}

\author[E. Kopeck\'a]{ Eva Kopeck\'a}
\address{Department of Mathematics\\
   University of Innsbruck\\
 A-6020 Innsbruck, Austria}

\email {eva.kopecka@uibk.ac.at}

\author[A. Paszkiewicz]{Adam Paszkiewicz}
  \address{ Faculty of Mathematics and Computer Science\\
   \L\'od\'z University\\
   Banacha 22 \\
   90-238 \L\'od\'z,
   Poland}
\email{adampasz@math.uni.lodz.pl}



\date{\today}

\begin{abstract}
Let $H$ be an infinite dimensional Hilbert space. We show that there exist three orthogonal projections $X_1, X_2, X_3$ onto closed subspaces of $H$ such that for every $0\ne z_0\in H$ there exist $k_1, k_2,\dots \in \{1,2,3\}$ so that the sequence of iterates defined by $z_n= X_{k_n} z_{n-1}$ does not converge in norm.
\end{abstract}

\maketitle


\section*{Introduction}

Let $K$ be a fixed natural number and let $X_1,X_2,\dots,X_K$ be a family of $K$
closed   subspaces of a Hilbert space  $H$.
 Let $z_0\in H$ and $k_1,k_2,\dots \in \{1,2,\dots,K\}$
be arbitrary. Consider the sequence of vectors $\{z_n\}_{n=1}^{\infty}$ defined by
\begin{equation}\label{iter}
z_{n}=X_{k_n}z_{n-1},
\end{equation}
where $X_k$ denotes also the orthogonal projection of $H$ onto the subspace  $X_k$.
The  sequence  $\{z_n\}$  converges weakly by a theorem of Amemiya and Ando \cite{AA}. If each projection appears in the sequence $\{P_{k_n}\}$ infinitely many times, then this limit is equal to the orthogonal projection of $z_0$   onto  $\bigcap_{i=1}^K X_i$.

If
$K=2$ then the sequence $\{z_n\}$ converges  even in norm according to a classical result of von Neumann \cite{N}. An elementary geometric proof of this theorem can be found in \cite{KR}.

If $K\geq 3$, then additional assumptions are needed to ensure the norm-convergence. That $\{z_n\}$ converges  if $H$ is finite dimensional  was originally proved by Pr\'ager \cite{Pr}; this also follows, of course,   from \cite{AA}.

If $H$ is infinite dimensional, but the sequence
$\{k_n\}$ is periodic, the sequence $\{z_n\}$ converges in norm according
to Halperin \cite{Ha}.
The result was generalized to quasiperiodic sequences by Sakai \cite{S}. Recall that
the sequence $\{k_n\}$ is  quasiperiodic if there exists $r\in\NN$ such that    $\{k_m,k_{m+1},\dots,k_{m+r}\}=\{1,2,\dots,K\}$ for each $m\in \NN$.

The question of norm-convergence if $H$ infinite dimensional, $K\geq 3$ and $\{k_n\}$ arbitrary, posed in \cite{AA}, was open for a long time. In 2012 Paszkiewicz \cite{P} constructed an ingenious example of 5 subspaces
of an infinite dimensional Hilbert space and of a sequence $\{z_n\}$ of the form (\ref{iter}) which does not
converge in norm. An important input towards the construction originates in Hundal's example (\cite{H}, see also \cite{K} and \cite{MR})
of two closed convex subsets of an infinite dimensional Hilbert space and a sequence
of alternating projections onto them which does not converge in norm.

E. Kopeck\'a and V. M\"uller   resolved in \cite{KM}  fully the question of Amemia and Ando. They  refined  Paszkiewicz's construction to get a rather complicated example of  three  subspaces of an infinite dimensional Hilbert space and of a sequence $\{z_n\}$ of the form (\ref{iter}) which does not
converge in norm. 

In Section~\ref{one} of this paper we considerably simplify the construction of \cite{KM}, resulting in Theorem~\ref{main}.  
Moreover, we strengthen the statement on existence of one badly behaved point to the  statement that {\em all points} (except for zero) in a  
Hilbert space $H$ can be badly behaved. Namely, in Theorem~\ref{mainmain}
we show that in every infinite dimensional Hilbert space $H$, there exist three orthogonal projections $X_1, X_2, X_3$ onto closed subspaces of $H$ such that for every $0\neq z_0\in H$ there exist $k_1, k_2,\dots \in \{1,2,3\}$ so that the sequence of iterates defined by $z_n= X_{k_n} z_{n-1}$ does not converge in norm.

If we allow five projections instead of just three, only two sequences of indices are needed to get divergence at every point. In Theorem~\ref{5pieces}
we show  that there exists a sequence $k_1,k_2,\dots\in\{1,2,3\}$ with the following property. For every infinite dimensional Hilbert space $H$
 there are closed subspaces $X,Y,X_1,X_2,X_3$ of $H$ so that if $0\neq z\in H$, and $u_0=Xz$ and $v_0=XYz$, then at least one of
  the sequences of iterates $\{u_n\}_{n=1}^{\infty}$ or $\{v_n\}_{n=1}^{\infty}$ defined by $u_n=X_{k_n}u_{n-1}$ and
 $v_n=X_{k_n}v_{n-1}$   does not converge in  norm.

The paper is organized as follows. 

In Section~\ref{words} we observe  that if 
we have  control over the number of appearances of a projection in a product,
then replacing the projection by another one, close in the norm to the original one, does not change the product much. 

Let $u$ and $v$ be two orthonormal vectors. 
In Section~\ref{one} we construct  two almost identical subspaces $X$ and $Y$
and a product of projections $\psi(u\vee v,X,Y)$ onto the span $u\vee v$ of $u$ and $v$, and the spaces $X$ and $Y$ so that $\psi(u\vee v,X,Y)u\approx v$.
Using the results of Section~\ref{words}, we ``glue" countably many copies of such triples of spaces together.  We thus  obtain three projections and a  divergent sequence    $\{z_n\}$ of iterates (\ref{iter})  which trail very closely along the quartercircles connecting  consecutive vectors of an orthonormal sequence.

In Section~\ref{all} we deduce from  the existence of three subspaces and one 
point the iterates of which do not converge the existence of other three subspaces
such that  for {\em any nonzero point} in the space certain iterates of the projections on these three subspaces do not converge.

\bigskip

{\bf Notation.} In the entire paper  $H$ is a Hilbert space; we will explicitly mention when we need $H$ to be infinite dimensional. Let  $B(H)$ be the space of bounded linear operators from $H$ to $H$.

For a closed subspace $X$ of $H$ we denote by $X$ also the orthogonal projection onto $X$.  For $M,N\subset H$ we denote by $\bigvee M$ the closed linear hull of $M$, and by $M\vee N$   the closed linear hull of $M\cup N$. Similarly, we use $\vee x$ and $x\vee y$ for $x,y\in H$.

For $m\in\NN$ let $\SSS_m$ be the free semigroup with generators $a_1,\dots,a_m$. 
If $\varphi=a_{i_r}\cdots a_{i_1}\in\SSS_m$ (for some $r\in\NN$ and $i_j\in\{1,\dots,m\}$)
 and $A_1,\dots,A_m\in B(H)$, then we write
$\varphi(A_1,\dots,A_m)=A_{i_r}\cdots A_{i_1}\in B(H)$. If $X_1,\dots,X_m$ are closed subspaces of $H$, then $\varphi(X_1,\dots,X_m)=X_{i_r}\cdots X_{i_1}\in B(H)$
is the respective product of orthogonal projections.

Denote by $|\varphi|=r$ the ``length" of the word $\varphi$, and by $|\varphi_i|$ the number of ``occurrences" of $a_i$ in the word $\varphi$. Then $\sum_{i=1}^m |\varphi_i|=|\varphi|=r$.

\section{Continuous dependence of words on the letters}\label{words}

Orthogonal projections are 1-Lipschitz mappings and so are their products. This allows  us to construct our examples ``imprecisely". If  for certain $z_0$ the sequence of iterates $\{z_n\}$ defined by (\ref{iter}) does not converge in norm, the same is true of such sequences of iterates starting from a point $w$   from  a small enough neighborhood of $z$.

 Now assume, we have  control on the number of appearances of  a letter $a$  in the  word $\phi$. If we plug in projections, substituting $A$ for $a$, or instead of $A$ a projection $B$ close enough in norm to $A$,   the resulting product does not change  much.

\begin{lemma}\label{wordcont} 
Let $\psi\in \SSS_n$ for some $n\in \NN$. Assume  $A_i,B_i,E\in  B(H)$, $i\in \{1,2,\dots,n\}$ are  contractions so that each $A_i$ commutes with $E$.   Then
$$
\|\psi(A_1,\dots, A_n)E -\psi(B_1,\dots, B_n)E\| \leq \sum_{1\leq i\leq n} |\psi_i| \ \|A_i E - B_i E\|.
$$
\end{lemma}
\begin{proof} We proceed by induction on $r =|\psi|$. For $r =0$ the statement is obvious as $\psi(A_1,\dots, A_n)=\psi(B_1,\dots, B_n)$ is the identity mapping on $H$. Let the statement  be valid upto some  $r$ and let $|\psi| = r+1$. Then $\psi =\phi a_j$ for some $\phi \in \SSS_n$ with $|\phi| =r$ and $j\in \{1,\dots, n\}$. Hence
\begin{equation}\notag
\begin{split}
\|&\psi(A_1,\dots, A_n)E -\psi(B_1,\dots, B_n)E\| \\
=&\|\phi(A_1,\dots, A_n)A_jE -\phi(B_1,\dots, B_n)B_jE\| \\
\leq& \|\phi(A_1,\dots, A_n)A_jE -\phi(B_1,\dots, B_n)A_jE\| \\
&+
\|\phi(B_1,\dots, B_n)A_jE -\phi(B_1,\dots, B_n)B_jE\| \\
\leq& \|\phi(A_1,\dots, A_n)E -\phi(B_1,\dots,B_n)E\|
+\|\phi(B_1,\dots, B_n)\|\cdot\|A_jE - B_jE\|\\
\leq& \|A_jE - B_jE\|+\sum_{1\leq i\leq n} |\phi_i| \ \|A_iE - B_i E\|\\
\leq &\sum_{1\leq i\leq n} |\psi_i| \ \|A_i E - B_i E\|,
\end{split}
\end{equation}
by induction, and since $A_jE=EA_j$, $\|A_j\|\leq 1$ and  $\|\phi(B_1,\dots, B_n)\|\leq 1$.
\end{proof}

The following two lemmata are straightforward colloraries of the above.
We include them for  easy reference.

 \begin{lemma}\label{cruc}
 Let $\psi\in \SSS_3$. Suppose $E,W,X,X',Y,Y',Z$ are subspaces of $H$ so that $W,X,Y\subset E$ and $X', Y'\perp E$.  Then
 $$
 \|\psi(W,X,Y)E-\psi(Z,X\vee X',Y\vee Y')E\|\leq |\psi_1|\cdot\|ZE-W\|.
 $$
 \end{lemma}
 \begin{proof}
 This is a corollary of Lemma~\ref{wordcont} for $n=3$, since $WE=W$, and $X=XE=(X\vee X')E$, and similarly, $Y=YE=(Y\vee Y')E$.
 \end{proof}
The next statement is a  corollary of Lemma~\ref{wordcont} for $E=H$.

\begin{lemma}\label{continuous}
Let $\phi\in \SSS_n$ and $A_1,\dots A_n$, $B_1,\dots, B_n$ be projections for some $n\in \NN$. Then
$$
\|\phi(A_1,\dots,A_n)-\phi(B_1,\dots,B_n)\|\leq |\phi|\cdot\max_{i=1,\dots,n}\|A_i-B_i\|.
$$
\end{lemma}

\section{Projectional iterates of a point  may diverge}\label{one}

Let $H$ be an infinite dimensional Hilbert space. According to \cite{KM},  there exist three closed subspaces $X_1,X_2,X_3\subset H$,      such that the sequence of iterates $\{w_n\}_{n=1}^{\infty}$ defined by $w_n=X_{k_n}w_{n-1}$ does not converge in  norm. In this section we simplify the proof of this statement.

The example is ``glued" together from finite dimensional blocks. In each of these blocks three subspaces and a finite product of projections are constructed  so that the product maps a given normalized vector $u$ with an arbitrary precision on a normalized vector $v$ orthogonal to $u$.

For $\eps>0$ let $k=k(\eps)$ be the smallest positive integer $k$ such that $(\cos\frac{\pi}{2k})^k>1-\e$. That is, if $u$ and $v$ are two orthonormal vectors, and we project $u$ consecutively onto the  lines $g_1,\dots, g_k$ dividing the right angle between $u$ and $v$ into $k$ angles of size $\frac{\pi}{2k}$, then we land at $v$ with error at most $\eps$

Projecting onto a  line $g_j$ can be arbitrarily well approximated by iterating projections between two subspaces intersecting at this particular line.
In the next lemma we call these spaces $W=u\vee v$ and $X'_j$. By introducing a small error, we modify $X_j'$ into $X_j$, to
  ensure that  $X_1\subset\dots\subset X_k$.
This will enable us later to replace the projections on each of the spaces $X_j$   by the product $(XYX)^{s(j)}$ of  projections onto the largest space $X=X_k$
and its suitable small variation $Y$.  Alltogether, instead of projecting  onto several spaces to get from $u$ to $v$,  we project only on three of them: $W$, $X$, and $Y$. We get a suitable product $\psi$ of these projections so that 
  $\psi(W,X,Y)u\approx v$.

 The  next lemma is  modified from \cite{P}, the proof is taken from \cite{KM}.

\begin{lemma}\label{1}
Let $\eps>0$. Then there exists $\phi\in\SS_{k(\eps)+1}$
with the following properties:

Suppose  $X$ is a subspace of $H$ so that $\dim X=\infty$,
and $u,v \in X$ are so that $\|u\|=\|v\|=1$ and $u\perp v$.
Then there exist  subspaces $X_1\subset \dots\subset X_{k(\eps)}\subset  X$ so that
  $\dim X_{j}=j+1$ for all $j\in \{1,\dots, k(\eps)\}$,   and
$$
\bigl\|\phi(W, X_1,\dots,X_{k(\e)})u-v\bigr\|<2\eps,
$$
where $W=u\vee v$.
\end{lemma}
\begin{proof}
Write $k:=k(\eps)$. Choose orthonormal vectors $z_0,z_1,\dots,z_{k-1}\in W^{\perp}\cap X$.

Let $\xi=\frac{\pi}{2k}$. For $j=0,\dots,k$, let $h_j=u\cos {j\xi}+v\sin{j\xi}$ be the points on the quarter circle connecting $h_0=u$ to $h_k=v$.
We construct inductively a rapidly decreasing sequence of nonnegative numbers $\al_0>\al_1>\cdots>\al_{k-1}>\al_k=0$ in the following way.
Choose $\al_0\in (0,1)$ arbitrarily. Let $1\le j\le k-1$ and suppose that
$\al_0,\dots,\al_{j-1}$ and subspaces $X_1\subset\cdots\subset X_{j-1}$
have already been constructed.
Set
$$
X'_j=\vee\{h_0+\al_0z_0,h_1+\al_1z_1,\dots,h_{j-1}+\al_{j-1}z_{j-1}, h_j\}.
$$
 Since $W\cap X'_j=\vee h_j$, we have
$(X'_jWX'_j)^rx\to (\vee h_j)x$  for each $x\in H$ as $r\to\infty$,   by \cite{N}. As both spaces are finite dimensional, there exists $r(j)\in\NN$ such that
$$
\bigl\|(X'_jWX'_j)^{r(j)}-(\vee h_j)\bigr\|<\frac{\eps
}{k}.
$$
Let $\al_j>0$ be so small that
\begin{equation}\label{line}
\bigl\|(X_jWX_j)^{r(j)}-(\vee h_j)\bigr\|<\frac{\eps}{k},
\end{equation}
where
$$
X_j=\vee\{h_0+\al_0z_0,h_1+\al_1z_1,\dots,h_{j-1}+\al_{j-1}z_{j-1}, h_j+\al_j z_j\}.
$$
Suppose that $X_1\subset X_2\subset\cdots\subset X_{k-1}$ have already been constructed. Set formally $\alpha_k=0$ and
$X_k=X'_k=\vee\{h_0+\al_0z_0,h_1+\al_1z_1,\dots,h_{k-1}+\al_{k-1}z_{k-1}, h_k\}$.
Find $r(k)\in\NN$ such that (\ref{line}) is true also for $j=k$. Then $v=h_k\in X_k$.
Let $\phi \in \SSS_{k+1}$ and $\psi\in \SSS_k$ be defined by
$$
\phi(c,b_1,\dots,b_k)=
(b_kcb_k)^{r(k)}\cdots (b_1cb_1)^{r(1)},
$$
and $\psi(a_1,\dots,a_k)=a_1\dots a_k$.
Then
\begin{equation}\notag
\begin{split}
\bigl\|\phi&(W,X_1,\dots,  X_k)-\vee h_k\dots\vee h_1\bigr\| \\
=&\bigl\|\psi((X_kWX_k)^{r(k)},\dots,(X_1WX_1)^{r(1)})- \psi(\vee h_k,\dots,\vee h_1)\bigr\| 
<k\cdot\frac{\eps}k,
\end{split}
\end{equation}
by Lemma~\ref{continuous}. Hence
\begin{equation}\notag
\begin{split}
\bigl\|\phi&(W,X_1,\dots,  X_k)u-v\bigr\| \\
\leq& \bigl\|\phi(W,X_1,\dots,  X_k)u-(\vee h_k\dots\vee h_1)u\bigr\|+
\bigl\|(\vee h_k\dots\vee h_1)u-v\bigr\|<2\eps.
\end{split}
\end{equation}
It follows from the construction, that the resulting word $\phi\in\SS_{k(\eps)+1}$
does not depend on the particular $X$, $u$, and $v$.

\end{proof}

Projections onto  an increasing family of $n$  finite dimensional spaces can be replaced
 by projections onto just two spaces: onto the largest space in the family
and onto  a suitable small variation of it.
The next lemma is  generalized from \cite{P}, its proof from \cite{KM}.
 
\begin{lemma}\label{ktotwo}
Let $k\in \NN$ and $\eps>0$, $\eta>0$, and $a>0$ be given. There exist natural
numbers $a<s(k)<s(k-1)<\cdots< s(1)$ with the following property.
Suppose
 $X_1\subset X_2\subset\cdots\subset X_k\subset X\subset E\subset  H$ are closed  subspaces
so that $X$ is separable and   $\dim X^{\perp}\cap E=\infty$.       Then there exists a closed subspace $Y\subset E$    such that $X\cap Y=\{0\}$, $\|X-Y\|<\eta$ and for each $j\in\{1,\dots,k\}$,
$$
\bigl\|(XYX)^{s(j)}-X_j\bigr\|<\eps.
$$
\end{lemma}
\begin{proof} We can assume $0<\eta<1$.
First we fix   $0<\beta_{k+1}<\eta/2$ and
choose $s(k)>a$ such that $1/(1+\beta_{k+1}^2)^{s(k)}<\eps$.
Next we
inductively choose numbers
$$
\beta_k, s(k-1),\beta_{k-1}, s(k-2),\dots, s(1),\beta_1
$$  such that
\begin{equation}\notag
\begin{split}
\beta_{k+1}&>\beta_{k}>\cdots>\beta_1>0, \\
a&<s(k)<s(k-1)<\cdots<s(1), \\
\frac{1}{(1+\beta_{j+1}^2)^{s(j)}}&<\e \mbox{ and }
\Bigl|\frac{1}{(1+\beta_j^2)^{s(j)}} - 1\Bigr|<\e
\end{split}
\end{equation}
for $j=k,\dots,1$. We show that these $s_j$'s are as required.

Let $\{e_i\}_{i\in I}$  be an at most countable orthonormal basis in $X$ such that there are index sets
$\emptyset=I_0\subset I_1\subset\cdots\subset I_k\subset I_{k+1}=I$  with the property, that $\{e_i\}_{i\in I_j}$  is an orthonormal basis in $X_j$ for
$j\in \{1,\dots,k\}$.   For $i\in I_j\setminus I_{j-1}$ define
$\gamma_i=\beta_j$. Let $\{w_i\}_{i\in I}$  be an orthonormal system  in $X^{\perp}\cap E$. We construct $Y$ as the closed linear span of the vectors $e_i+\gamma_iw_i$, $i\in I$.
Note that if $Y$ is constructed in this way, we have for $m\in \NN$ and $i\in I$,
\begin{equation}\label{itera}
(XYX)^me_i=\frac{e_i}{(1+\gamma_i^2)^m}.
\end{equation}

If $x=\sum_{i\in I} a_ie_i\in X$, then by (\ref{itera}),
\begin{equation}\label{estimate}
\begin{split}
&\bigl\|(XYX)^{s(j)}x-X_jx\bigr\|^2
=\bigl\|\sum_{i\in I}a_i\frac{e_i}{(1+\gamma_i^2)^{s(j)}}-\sum_{i\in I_j}a_ie_i \bigr\|^2 \\
&=\sum_{i\in I_j}a^2_i\left(1-\frac{1}{(1+\gamma_i^2)^{s(j)}}\right)^2+\sum_{i\in I\setminus I_j} a^2_i\frac{1}{(1+\gamma_i^2)^{2s(j)}} \\
&\leq \eps^2\sum_{i\in I} a_i^2=\eps^2\|x\|^2.
\end{split}
\end{equation}
For any $z\in H$ we have
$$
(XYX)^{s(j)}z-X_jz=(XYX)^{s(j)}(Xz)-X_j(Xz),
$$
since $X_j\subset X$. Hence by (\ref{estimate}) for $j\in\{1,\dots,k\}$,
$$
\bigl\|(XYX)^{s(j)}-X_j\bigr\|<\e.
$$
 It is easy to see that $\|X-Y\|<\eta$. Indeed, for any $0\neq z\in H$,
\begin{equation}\notag
\begin{split}
\|Xz-Yz\|^2=& \|\sum_{i\in I}\langle e_i,z\rangle e_i-\sum_{i\in I}\frac 1{1+\gamma_i^2}\langle e_i+\gamma_i w_i,z\rangle (e_i+\gamma_i w_i)\|^2\\
\leq&\sum_{i\in I}\frac 1{(1+\gamma_i^2)^2}[\gamma_i^2\langle e_i,z\rangle-
\gamma_i\langle  w_i,z\rangle]^2 \\
&+\frac{\eta^2}4\sum_{i\in I}[\frac 1{1+\gamma_i^2}\langle e_i+\gamma_i w_i,z\rangle]^2 \\
\leq & 2(\eta/2)^4 \|z\|^2+2(\eta/2)^2\|z\|^2+(\eta^2/4)\|z\|^2 \\
< &\eta^2\|z\|^2
\end{split}
\end{equation}
We have used that  $0< \gamma_i<\beta_{k+1}<\eta/2<1$, and in the third line the estimate $(a-b)^2\leq 2a^2+2b^2$. Since $0< \gamma_i$ for all $i\in I$, we have  that $X\cap Y=\{0\}$.
\end{proof}

Again, let $u$ and $v$ be orthornormal, $W=u\vee v$,  and $\eps>0$. In the next lemma 
we construct a  word $\psi$ and a two ``almost parallel" spaces $X$ and $Y$ so that 
$\bigl\|\psi(W, X,Y)u-v\bigr\|<2\eps$. Importantly, the number of appearances of $W$ in $\psi$ depends {\em only} on $\eps$. On the contrary, the closer together the spaces $X$ and $Y$ are, the  larger is their  number of appearances in $\psi$.

\begin{lemma}\label{N}
For every  $\eps>0$, there exists $N=N(\eps)$, so that for every $\eta>0$, there exists $\psi\in \SSS_3$ so that $|\psi_1|\leq N$
with the following property.

Given subspaces
 $X\subset E\subset  H$
so that $X$ is separable and   $\dim X=\dim X^{\perp}\cap E=\infty$, and $u,v \in X$ are so that $\|u\|=\|v\|=1$ and $u\perp v$,
there exists a subspace $Y\subset E$    such that $X\cap Y=\{0\}$, $\|X-Y\|<\eta$, and
$$
\bigl\|\psi(W, X,Y)u-v\bigr\|<3\eps,
$$
where $W=u\vee v$.
\end{lemma}
\begin{proof}
Let $\eps>0$ be given. Let  $\phi\in\SS_{k(\eps)+1}$ be as in Lemma~\ref{1}, and
 $N=|\phi_1|$. Let $\eta>0$ be given.
For $k=k(\eps)$, $\eps$ replaced by $\eps/|\phi|$,  the given   $\eta$, and $a=1$ choose the natural numbers
$s(k)<s(k-1)<\dots<s(1)$ according to Lemma~\ref{ktotwo}.
To define the word $\psi$, replace for each $i\in \{2,\dots, k+1\}$ the letter
$a_i$ in $\phi$ by $(a_2a_3a_2)^{s(i-1)}$. With a slight abuse of notation, but certainly more understandably,
$$
\psi(W,X,Y)=
\phi(W,(XYX)^{s(1)},\dots, (XYX)^{s(k)}).
$$
Clearly, $|\psi_1|=|\phi_1|=N$.

Let $u,v\in X\subset E$ as in the lemma be given.
According to Lemma~\ref{1}, there exist subspaces
$X_1\subset \dots\subset X_{k}\subset  X$ so that
$$
\bigl\|\phi(W, X_1,\dots,X_{k})u-v\bigr\|<2\eps,
$$
where $W=u\vee v$. By Lemma~\ref{ktotwo} there exists a subspace $Y\subset E$    such that $X\cap Y=\{0\}$, $\|X-Y\|<\eta$ and for each $j\in\{1,\dots,k\}$,
$$
\bigl\|(XYX)^{s(j)}-X_j\bigr\|<\eps/|\phi|.
$$
Then
\begin{equation}\notag
\begin{split}
\bigl\|\psi&(W,X,Y)u-v\bigr\| \\
=&\|\phi(W,(XYX)^{s(1)},\dots, (XYX)^{s(k)})u-v\| \\
\leq &
\bigl\|\phi(W,(XYX)^{s(1)},\dots, (XYX)^{s(k)})u-\phi(W,X_1,\dots,X_k)u\bigr\|
\\
&+\bigl\|\phi(W,X_1,\dots,X_k)u-v\bigr\|<3\eps
\end{split}
\end{equation}
by Lemma~\ref{continuous}.
\end{proof}

Next we show that we have  relative freedom of choice for the three spaces 
verifying  $\psi(W,X,Y)u\approx v$. 

\begin{lemma}\label{delta}
For every  $\eps>0$, there exists $\delta=\delta(\eps)$, so that for every $\eta>0$, there exists $\psi\in \SSS_3$
with the following property.

Given subspaces
 $X\subset E\subset  H$
so that $X$ is separable and   $\dim X=\dim X^{\perp}\cap E=\infty$, and $u,v \in X$ are so that $\|u\|=\|v\|=1$ and $u\perp v$,
there exists a subspace $Y\subset E$    such that $X\cap Y=\{0\}$, $\|X-Y\|<\eta$ with the following property. If $W=u\vee v$ and   $X',Y', Z$ are subspaces such that $X',Y'\subset E^{\perp}$ and $\|W-ZE\|<\delta$ then
$$
\bigl\|\psi(Z, X\vee X',Y\vee Y')u-v\bigr\|<4\eps.
$$
\end{lemma}
\begin{proof}
Given an $\eps>0$, choose $N\in \NN$ according to Lemma~\ref{N} and put $\delta=\eps/N$. For these $\eps$ and $N$ and a given $\eta$ choose  $\psi$ according to Lemma~\ref{N}, and for a given subspace $X$ choose also $Y$ according to this lemma.
Let $X',Y', Z$ be as above.  Lemma~\ref{cruc} implies, that
\begin{equation}\notag
\begin{split}
\bigl\|\psi(Z,& X\vee X',Y\vee Y')u-v\bigr\| \\ \leq &
\bigl\|\psi(Z, X\vee X',Y\vee Y')u- \psi(W, X,Y)u   \bigr\| +
\bigl\|\psi(W, X,Y)u-v\bigr\| \\
\leq &
|\psi_1|\|W-ZE\|+3\eps\leq N\eps/N+3\eps=4\eps.
\end{split}
\end{equation}
\end{proof}

Now we glue the finite dimensional steps together.
Given an orthonormal sequence $\{e_i\}_{i=1}^{\infty}$ with an infinite dimensional orthocomplement,
we construct three spaces $X,Y,Z$  and words $\Psi_i\in \SSS_3$ 
so that $\Psi_i(Z,X,Y)e_i\approx e_{i+1}$  for all $i\in \NN$. 

\begin{lemma}\label{constr}
For any  $\eps_i>0$, $i=1,2, \dots$, there exist $\Psi_i\in \SSS_3$ with the following property.

Suppose $\{e_i\}_{i=1}^{\infty}$ is an orthonormal sequence in $H$ with
an infinite dimensional orthogonal complement.
Then there are
three closed subspaces $X,Y,Z\subset H$    so that
$$
\bigl\|\Psi_i(Z, X,Y)e_i-e_{i+1}\|<4\eps_i.
$$
\end{lemma}
\begin{proof}
Use Lemma~\ref{delta} to define $\delta_i=\delta(\eps_i)$.
Put $\eta_i=\min\{\delta_{i-1}, \delta_{i+1}\}/2$, where $\delta_0=1$.
Again, use  Lemma~\ref{delta} to choose $\psi_i\in \SSS_3$.
For even $i\in \NN$ put $\Psi_i=\psi_i$, for the odd $i$
define $\Psi_i(Z,X,Y)=\psi_i(Y,X,Z)$.

Let $E_i$, $i\in \NN$ be closed infinite dimensional subspaces of $H$ so that
\begin{equation}\notag
\begin{split}
e_i, e_{i+1}&\in E_i, \\
\vee e_{i+1}&=E_iE_{i+1}=E_{i+1}E_i, \mbox{ and} \\
E_i&\perp E_j \mbox{ if } |i-j|\geq 2.
\end{split}
\end{equation}
Fix also closed subspaces $X_i\subset E_i$, so that
$e_i, e_{i+1}\in X_i$, and $\dim X_i=\dim( X_i^{\perp}\cap E_i)=\infty$.

By Lemma~\ref{delta}, there exist
  closed subspaces
$Y_i\subset E_i$ so that
\begin{equation}\label{est}
\bigl\|\psi_i(Z_i, X_i\vee X',Y_i\vee Y')e_i-e_{i+1}\bigr\|<4\eps,
\end{equation}
whenever $W_i=e_i\vee e_{i+1}$ and   $X',Y', Z_i$ are subspaces such that $X',Y'\subset E_i^{\perp}$ and $\|W_i-Z_iE_i\|<\delta_i$. Put $Y_0=\vee e_1$ and
$$
X=\bigvee_{i=1}^{\infty} X_i,\  Y=\bigvee_{k=0}^{\infty} Y_{2k},\ Z=\bigvee_{k=0}^{\infty} Y_{2k+1}.
$$
Then for any $i\in \NN$ we have $X=X_i\vee X_i'$ for a suitable $X_i'\perp E_i$.
Further we distinguish between two cases: $i$ is even and $i$ is odd.

If $i=2k$ is even, then also $Y=Y_i\vee Y_i'$ for a suitable $Y_i'\perp E_i$.
Since $ZE_i=(Y_{i-1}\vee Y_{i+1})E_i$, we have
\begin{equation}\notag
\begin{split}
\|W_i-ZE_i\|&=\|(X_{i-1}\vee X_{i+1})E_i-(Y_{i-1}\vee Y_{i+1})E_i\| \\
&=\|(X_{i-1}+ X_{i+1})E_i-(Y_{i-1}+ Y_{i+1})E_i\| \\
&=\|(X_{i-1}-Y_{i-1})E_i+(X_{i+1}- Y_{i+1})E_i\| \\
&\leq \|X_{i-1}-Y_{i-1}\|+\|X_{i+1}-Y_{i+1}\|\leq \eta_{i-1}+\eta_{i+1} \\
&\leq \min\{\delta_{i-2}+\delta_i\}/2+\min\{\delta_{i}+\delta_{i+2}\}/2\leq \delta_i.
\end{split}
\end{equation}
Hence by (\ref{est}),
$$
\bigl\|\psi_i(Z, X,Y)e_i-e_{i+1}\|<4\eps_i.
$$
If $i=2(k-1)$ is odd, then $Z=Y_i\vee Y_i'$ for a suitable $Y_i'\perp E_i$,
and similarly as above one can show that
$$
\|W_i-YE_i\|\leq \delta_i.
$$
Hence by (\ref{est}),
$$
\bigl\|\psi_i(Y, X,Z)e_i-e_{i+1}\|<4\eps_i.
$$

\end{proof}

Finally we prove the main result of this section, Theorem~2.6 of \cite{KM}.

\begin{theorem}\label{main}  There exists a sequence $k_1,k_2,\dots\in\{1,2,3\}$ with the following property.
If  $H$ is an infinite dimensional Hilbert space and $0\neq w_0\in H$, then
there exist three closed subspaces $X_1,X_2,X_3\subset H$,      such that the sequence of iterates $\{w_n\}_{n=1}^{\infty}$ defined by $w_n=X_{k_n}w_{n-1}$ does not converge in  norm. (Here $X_n$ denotes also the orthogonal projection onto $X_n$.) More precisely,
the spaces $X_1,X_2,X_3\subset H$ intersect only at the origin, and the sequence
$\{w_n\}$, although weakly convergent to the origin, stays in norm bounded
away from zero.
\end{theorem}
\begin{proof}
For $\eps_i=9^{-i}$ choose  $\Psi_i$ as in Lemma~\ref{constr}. Let $e_1=w_0/|w_0|$. Fix an orthonormal  sequence $\{e_i\}_{i=1}^{\infty}$ with an infinite dimensional complement.  Choose  closed subspaces
 $X,Y,Z\subset H$ according to  Lemma~\ref{constr} and call them $X_1, X_2, X_3$. Write for short $A_i=\Psi_i(Z,X,Y)$.
 Since $\|A_i\|\leq 1$, we have by induction
 \begin{equation}\notag
 \begin{split}
 \|A_n&A_{n-1}\dots A_1 e_1 -e_{n+1}\| \\
 &\leq \|A_nA_{n-1}\dots A_2(A_1 e_1 -e_2)\|+ \|A_nA_{n-1}\dots A_2e_2-e_{n+1}\| \\
 &\leq 4\eps_1+\|A_nA_{n-1}\dots A_3(A_2 e_2 -e_3)\|+ \|A_nA_{n-1}\dots A_3e_3-e_{n+1}\| \\
& \leq 4\eps_1+4\eps_2+\dots+\|A_n e_n-e_{n+1}\|\leq 4(9^{-1}+\dots+9^{-n})\leq 1/2
 \end{split}
 \end{equation}
 for all $n\in \NN$. Since $\{e_i\}$ is an orthonormal sequence, the norm-limit
 $\lim_{n\to \infty} A_nA_{n-1}\dots A_1 e_1$ does not exist.
\end{proof}

\section{All points  may have  diverging projectional iterates}\label{all}

In this section we strengthen the statement on existence of one badly behaved point to the  statement that {\em all points} (except for zero) in a  
Hilbert space $H$ can be badly behaved. Namely, in Theorem~\ref{mainmain}
we show that in every infinite dimensional Hilbert space $H$, there exist three orthogonal projections $X_1, X_2, X_3$ onto closed subspaces of $H$ such that for every $0\neq z_0\in H$ there exist $k_1, k_2,\dots \in \{1,2,3\}$ so that the sequence of iterates defined by $z_n= X_{k_n} z_{n-1}$ does not converge in norm.

If we allow five projections instead of just three, only two sequences of indices are needed to get non-convergence at every point. In Theorem~\ref{5pieces}
we show that there exists a sequence $k_1,k_2,\dots\in\{1,2,3\}$ with the following property. For every infinite dimensional Hilbert space $H$
 there are closed subspaces $X,Y,X_1,X_2,X_3$ of $H$ so that if $0\neq z\in H$, and $u_0=Xz$ and $v_0=XYz$, then at least one of
  the sequences of iterates $\{u_n\}_{n=1}^{\infty}$ or $\{v_n\}_{n=1}^{\infty}$ defined by $u_n=X_{k_n}u_{n-1}$ and
 $v_n=X_{k_n}v_{n-1}$   does not converge in  norm.
 
First we prove an auxiliary statement: in an infinite dimensional Hilbert space there always is an infinite dimensional subspace of points which are badly behaved with respect to projections on suitably chosen three subspaces.  

\begin{lemma}\label{infix} There exists a sequence $k_1,k_2,\dots\in\{1,2,3\}$ with the following property. Suppose  $X$ is an infinite dimensional subspace of a Hilbert space $H$ so that the Hilbert dimension of $X^\perp$ is at least as large as that of $X$.
Then there exist three closed subspaces $X_1,X_2,X_3\subset H$, so that for each
 $0\neq z_0\in X$    the sequence of iterates $\{z_n\}_{n=1}^{\infty}$ defined by $z_n=X_{k_n}z_{n-1}$ does not converge in  norm.
\end{lemma}
\begin{proof} Let $k_1,k_2,\dots\in\{1,2,3\}$ be as in Theorem~\ref{main}. Let $\{w^{\lambda}\}_{\lambda\in \Lambda}$ be an
orthonormal basis of $X$. Let $F_{\lambda}$, $\lambda \in \Lambda$ be pairwise orthogonal closed infinite dimensional subspaces of $X^\perp$. Define $E_{\lambda}=w_{\lambda}\vee F_{\lambda}$.
In each $E_{\lambda}$ choose closed subspaces $X^{\lambda}_1,X^{\lambda}_2,X^{\lambda}_3\subset E_{\lambda}$  as in Theorem~\ref{main}.
For $j\in \{1,2,3\}$ define
$X_j=\bigvee_{\lambda\in \Lambda} X^{\lambda}_j$.

Let $0\neq z_0=\sum_{\lambda\in \Lambda}t_{\lambda}w^{\lambda} \in X$ be given. Choose $\alpha\in \Lambda$ so that
$t_{\alpha}\neq 0$ and write $z_0=tw_0+u_0$, where $t=t_\alpha$, and $w_0=w^{\alpha}$. Since the spaces $E_\lambda$ are pairwise orthogonal, the iterates of the point
$z_0$ can be easily expressed using the iterates of $w_0$ and $u_0$, namely $z_n=tw_n+u_n$. Moreover,
$$
\|z_n\|^2=t^2\|w_n\|^2+\|u_n\|^2\geq t^2\|w_n\|^2,
$$
since $w_n\in E_{\alpha}$ and $u_n\in E_{\alpha}^\perp$. Hence, the norm of the sequence $\{z_n\}$
stays bounded away from zero. As $\{z_n\}$ converges weakly to zero by \cite{AA},
it does not converge in norm.
\end{proof}

Next we will construct five subspaces of an infinite dimensional Hilbert space $H$, 
so that for each $0\neq x\in H$ a certain sequence of iterates of projections of $x$ on these five spaces does not converge in norm. We will use the following elementary observation.

\begin{lemma}\label{nonzero}
Let $H$ be a Hilbert space and let $X$   be a closed  subspace of $H$ so that  the Hilbert dimensions of $X$ and $X^{\perp}$ are the same.  Then there exists a closed subspace $Y$ of $H$ so that $Xz\neq 0$ or $XYz\neq 0$ for each $0\neq z\in H$. 
\end{lemma}
\begin{proof}
Let $\{e_{\lambda}\}_{\lambda\in \Lambda}$ and $\{f_{\lambda}\}_{\lambda\in \Lambda}$ be the orthonormal bases of $X$ and $X^{\perp}$ respectively.
Then $\{e_{\lambda}+f_{\lambda}\}_{\lambda\in \Lambda}$ is an orthogonal sequence;
let $Y$ be its closed linear span. Suppose $0\neq z\in H$ is so that $Xz=0$.
Then $z=\sum\langle f_{\lambda},z\rangle f_{\lambda}$  with some $\langle f_{\alpha},z\rangle\neq 0$. Hence
$$
XYz=X\left(\sum_{\lambda\in \Lambda} \langle f_{\lambda},z\rangle \frac {e_\lambda+f_\lambda}2\right)=\frac 12
\sum_{\lambda\in \Lambda}\langle f_{\lambda},z\rangle e_\lambda\neq 0.
$$
\end{proof}

\begin{theorem}\label{5pieces}
Let $H$ be an infinite dimensional Hilbert space. Then there exist five closed subspaces
of $H$ so that for every $0\neq z\in H$ some sequence of iterates of $z$ defined by the orthogonal projections on those subspaces does not converge in norm.

More precisely, there exists a sequence $k_1,k_2,\dots\in\{1,2,3\}$ with the following property. Every infinite dimensional Hilbert space $H$ has
 closed subspaces $X,Y,X_1,X_2,X_3$  so that if $0\neq z\in H$, and $u_0=Xz$, $v_0=XYz$, then at least one of
  the sequences of iterates $\{u_n\}_{n=1}^{\infty}$ or $\{v_n\}_{n=1}^{\infty}$ defined by $u_n=X_{k_n}u_{n-1}$,
 $v_n=X_{k_n}v_{n-1}$   does not converge in  norm.
\end{theorem}
\begin{proof}
Let $k_1,k_2,\dots\in\{1,2,3\}$ be as in Lemma~\ref{infix}.
Choose a closed subspace $X$ of $H$ so that  the Hilbert dimensions of $X$ and $X^{\perp}$ are the same. Choose $X_1, X_2, X_3$ according to Lemma~\ref{infix}
and $Y$ according to Lemma~\ref{nonzero}.
\end{proof}

A  stronger result with only three subspaces  can be obtained in a non-constructive  way.
First we will show that if the iterates of three projections  of a certain point 
do not converge in norm, then there is a closed infinite dimensional subspace of such points. We recall a couple of elementary facts about projections we will need in the proof.

Let $H$ be a Hilbert space. A continuous linear mapping $P:H\to H$ is called a projection if $P^2=P$. It is an orthogonal projection if the range and the kernel of $P$ are orthogonal. It is easy to see that a projection is orthogonal if and only if it is self adjoint.

Let $X,Z\subset H$ be closed subspaces so that $Xz\in Z$ for each $z\in Z$. Then $XZ=ZX$ is the orthogonal projection onto $X\cap Z$.

Conversely, assume that the orthogonal projections onto $X$ and $Z$ commute. Then again, $XZ$ is the orthogonal projection onto $X\cap Z$.

\begin{lemma}\label{stronger}
Let $H$ be a Hilbert space and let $X_1, X_2, X_3\subset H$ be three of its closed subspaces.
Suppose $w_0\in H$ and $k_1,k_2,\dots\in\{1,2,3\}$ are so that the sequence defined by $w_{n}=X_{k_n}w_{n-1}$ does not converge in norm. Then there exists a closed infinite dimensional subspace $L$ of $H$ and closed subspaces $Y_1, Y_2, Y_3\subset L$ so that 
  for every $0\neq u_0\in L$ there exist $j_1,j_2,\dots\in\{1,2,3\}$  so that the sequence defined by $u_{n}=Y_{j_n}u_{n-1}$ does {\em not} converge in norm.
\end{lemma}

\begin{proof} 
Let $Z$ be the set of all $z_0\in H$ so that for every sequence $j_1,j_2,\dots\in\{1,2,3\}$  the sequence defined by $z_{n}=X_{j_n}z_{n-1}$ does  converge in norm.
Clearly,  $Z$ is a linear subspace; we observe  that it is also closed. Indeed, assume that $\lim_{i\to \infty} y_0^i=y_0$ for some $y_0^i\in Z$ and that $j_1,j_2,\dots\in\{1,2,3\}$. Define recursively the sequences $y_{n}=X_{j_n}y_{n-1}$ and $y^m_{n}=X_{j_n}y^m_{n-1}$ for $m\in \NN$.
For each $\eps>0$ choose $m\in \NN$ so that $\|y_0^m-y_0\|<\eps$ and $N\in \NN$ so that 
$\|y_k^m-y_l^m\|<\eps$ for all $k,l>N$.
Then $\|y_k-y_l\|\leq \|y_k-y_k^m\|+\|y_k^m-y_l^m\|+\|y_l^m-y_l\|<3\eps$, since the composition of projections is $1$-Lipschitz. Hence the sequence $\{y_{n}\}_{n=0}^{\infty}$ is Cauchy and therefore convergent, and $y_0\in Z$.

From the definition of $Z$ it follows that if $z\in Z$, then also  $X_i z \in Z$ for $i\in \{1,2,3\}$. Thus $X_iZ=ZX_i$ is an orthogonal   projection onto $X_i\cap Z$. Define $L=Z^{\perp}$. Then $L= Id_H -Z$ for the orthogonal projection onto $L$, hence $L$ commutes with $X_i$ for  $i\in \{1,2,3\}$. Therefore 
$LX_i=X_iL$ is the orthogonal projection onto $Y_i=L\cap X_i$.

Assume   $0\neq u_0\in L$. Since $L\cap Z=\{0\}$,  there exist $j_1,j_2,\dots\in\{1,2,3\}$  so that the sequence defined by 
$$
u_{n}=X_{j_n}u_{n-1}=Y_{j_n}u_{n-1}
$$
 does  not converge in norm, hence $L$ is infinite dimensional by \cite{Pr}.\end{proof}

We are now ready  to prove the main theorem of our paper.

\begin{theorem}\label{mainmain}
 Every infinite dimensional Hilbert space $H$ contains three closed subspaces $X_1, X_2, X_3$ with the following property. For every $0\neq w_0\in H$ there is a sequence   $k_1, k_2,\dots \in \{1,2,3\}$ so that the sequence of iterates defined by $w_n= X_{k_n} w_{n-1}$ does not converge in norm.
\end{theorem}
\begin{proof}
 Assume first $H$ is separable. According to Theorem~\ref{main}, there exist  $w_0\in H$ and $k_1,k_2,\dots\in\{1,2,3\}$ so that the sequence defined by $w_{n}=X_{k_n}w_{n-1}$ does not converge in norm. Let $L$ be the infinite dimensional subspace of $H$ obtained in Lemma~\ref{stronger}. Since $H$ is separable, $L$ and $H$ are isometric, and the required subspaces of $L$, alias $H$, exist.
 
For a  non-separable Hilbert space  $H$  we choose pairwise orthogonal, separable, infinite dimensional, closed spaces $H_{\lambda}\subset H$ so that 
$$
H=\bigvee_{\lambda\in \Lambda}H_{\lambda}.
$$
In each $H_{\lambda}$ we choose closed subspaces $X^{\lambda}_1, X^{\lambda}_2,X^{\lambda}_3$ according to the separable case of the theorem. We define 
$$
X_i=\bigvee_{\lambda\in \Lambda}X_i^{\lambda}, \  i\in \{1,2,3\}.
$$  
If $0\neq w_0\in H$, then $w_0=\sum_{\lambda\in \Lambda} w^{\lambda}_0$ for some $w_\lambda\in H_{\lambda}$. Since  $0\neq w^{\alpha}_0$ for some $\alpha\in \Lambda$,
there exist $k_1, k_2,\dots \in \{1,2,3\}$ so that 
the sequence of iterates defined by $w^{\alpha}_n=X^{\alpha}_{k_n} w^{\alpha}_{n-1}=X_{k_n} w^{\alpha}_{n-1}$ does not converge in norm. As 
$$
w_n= X_{k_n} w_{n-1}=\sum_{\lambda\in \Lambda}w^{\lambda}_n,
$$
$\|w_n-w_m\|\geq \|w^{\alpha}_n-w^{\alpha}_m\|$ and the sequence $\{w_n\}_{n=0}^{\infty}$ does not converge either.
\end{proof}

\subsection*{Acknowledgements}

Eva Kopeck\'a was partially supported by  Grant FWF-P23628-N18.
Adam Paszkiewicz was partially supported by Grant UMO-2011/01/B/ST1/03994.

\end{document}

\section*{Appendix: Projections}

Let $H$ be a Hilbert space. A continuous linear mapping $P:H\to H$ is called a projection if $P^2=P$. It is an orthogonal projection, if the range and the kernel of $P$ are orthogonal, that is, $\langle Pu, v-Pv\rangle= \langle u-Pu, Pv\rangle=0$ for all $u,v\in H$. It is easy to see, that a projection is orthogonal if and only if it is self adjoint.

\begin{lemma}
Let $X,Z\subset H$ be closed subspaces so that $Xz\in Z$ for each $z\in Z$. Then $XZ=ZX$ is the orthogonal projection onto $X\cap Z$.
\end{lemma}
\begin{proof}
If $v\in H$, then $XZv\in X\cap Z$ and
$$
XZ(XZv)=XXZv=XZv,
$$
hence $XZ$ is a projection. Suppose $w\in X\cap Z$. Then $XZw=Xw=w$, hence $XZ$
is a projection onto $X\cap Z$. For any $v\in H$
$$
\langle XZv-v,w\rangle=\langle (XZv-Zv)+(Zv-v),w\rangle=0,
$$
hence $XZ$ is an orthogonal projection onto $X\cap Z$, and it is self adjoint.
Then for any $u,v\in H$
$$
\langle XZu,v\rangle=\langle u,XZv\rangle=\langle Xu,Zv\rangle=\langle ZXu,v\rangle,
$$
since $X$ and $Z$ are also self-adjoint. Hence $XZ$=$ZX$.
\end{proof}

\begin{lemma}
Let $X,Z\subset H$ be closed subspaces so that the orthogonal projections onto $X$ and $Z$ commute. Then $XZ$ is the orthogonal projection onto $X\cap Z$.
\end{lemma}
\begin{proof}
Since 
$$
(XZ)^2=XZXZ=XZZX=XZX=XXZ=XZ,
$$
$XZ$ is a projection into $X\cap Z$. Suppose $w\in X\cap Z$. Then $XZw=Xw=w$, hence $XZ$
is a projection onto $X\cap Z$.

Since $X$ and $Z$ are self adjoint, for any $u,v\in H$
$$
\langle XZu,v\rangle=\langle Zu,Xv\rangle=\langle u,ZXv\rangle=\langle u,XZv\rangle,
$$
hence $XZ$ is also self adjoint a it is therefore an orthogonal projection onto $X\cap Z$. 
\end{proof}
\end{document}